\documentclass[american, 11pt]{amsart}
\usepackage{amsmath, graphicx, amsthm, amssymb,faktor}
\usepackage[left=2.5cm,right=2cm, bottom= 3cm]{geometry}
\usepackage{pb-diagram, srcltx}
\usepackage{amsfonts}
\usepackage{amscd}
\usepackage{euscript}
\usepackage{hyperref}
\usepackage{fancyhdr}
\usepackage{float}
\usepackage{subfigure}
\usepackage[utf8]{inputenc}
\usepackage{enumerate}
\usepackage[curve]{xypic}

\usepackage{color}

\newcounter{warn}[page]
\newcommand{\danger}{${\color{red}\triangle}\llap{\raisebox{.3ex}%
{\tiny!\hspace{1.45ex}}}$}
\newcommand{\warning}[1]{%
\raisebox{.01em}[0em]{\danger\ifnum\value{warn} > 1%
\tiny\bf\arabic{warn}\fi}%
\marginpar{\color{red}\tiny{\ifnum\value{warn} > 1\tiny\bf\arabic{warn}:\fi}\tiny #1}%
\stepcounter{warn}}

\newtheorem{theorem}{Theorem}[section]
\newtheorem{definition}[theorem]{Definition}

\newtheorem{corollary}[theorem]{Corollary}
\newtheorem{proposition}[theorem]{Proposition}
\newtheorem{lemma}[theorem]{Lemma}

\theoremstyle{remark}
\newtheorem{remark}{Remark}

\renewcommand{\emptyset}{\varnothing}

\newcommand{\Z}{\mathbb{Z}}

\newcommand{\R}{\mathbb{R}}
\newcommand{\C}{\mathbb{C}}
\renewcommand{\S}{\mathbb{S}}

\newcommand{\eval}{\mathrm{ev}}
\newcommand{\Gen}[1][\star]{{\mathcal{L}_{#1}}}
\newcommand{\Loops}{\mathop{\Omega}}
\newcommand{\Orbits}{\mathcal{I}}
\newcommand{\CovOrbits}{\tilde{\Orbits}}
\newcommand{\Crit}{\mathrm{Crit}}
\newcommand{\U}{\mathcal{U}}
\newcommand{\M}{\mathcal{M}}

\newcommand{\glueto}{\overset{\sharp}{\leftrightarrow}}

\newcommand{\omegastd}{\omega_{\mathrm{std}}}
\newcommand{\Id}{\mathrm{Id}}

\author{Jean-Fran\c cois Barraud}
\author{Lara Simone Su\'arez}
\thanks{The second author was supported by the funding from the European Community’s
Seventh Framework Progamme ([FP7/2007-2013] [FP7/2007-2011]) under grant agreement no [258204].}

\email{barraud@math.univ-toulouse.fr}
\email{lara.suarez@mail.huji.ac.il}

\address{Institut de math\'{e}matiques de Toulouse, Universit\'{e} Paul Sabatier – Toulouse
III, 118 route de Narbonne, F-31062 Toulouse Cedex 9, FRANCE}

\address{School of Mathematical Sciences, Tel Aviv University, ISRAEL.}

\begin{document}

\title{The fundamental group of a rigid Lagrangian Cobordism}

\begin{abstract}
 In this article we extend the construction of the Floer fundamental group to the monotone Lagrangian setting and use it to study the fundamental group
 of a Lagrangian cobordism $W\subset (\C\times M, \omega_{st}\oplus
 \omega)$ between two Lagrangian submanifolds $L, L'\subset ( M,
 \omega)$. We show that under natural conditions the inclusions $L,L'\hookrightarrow W$
induce surjective maps $\pi_{1}(L)\twoheadrightarrow\pi_{1}(W)$, $\pi_{1}(L')\twoheadrightarrow\pi_{1}(W)$ and when the previous maps are injective then $W$ is an h-cobordism. 
\end{abstract} 
\maketitle

\section{Introduction}
In \cite{Pi1} the first author recovered the fundamental group of a monotone symplectic manifold $(M, \omega)$ by studying the moduli spaces
of ``augmentations'' (see Definition \ref{def: augmentation}). The first objective of this paper is to extend this construction to the Lagrangian setting~: when no holomorphic disc can arise on the Lagrangian
submanifold, the Lagrangian and Hamiltonian settings turn out to be very similar, and the same ideas allow to recover the fundamental group of a Lagrangian submanifold from Floer theoretical objects.

The second objective of this paper is to study the fundamental group of a Lagrangian cobordism, in particular a monotone or weakly exact cobordism with two ends.

Let $L, L' \subset (M, \omega)$ be two compact, connected Lagrangian submanifolds. Denote by $\tilde{M} = \C \times M$ the symplectic manifold with the two-form $\tilde{\omega}= dx\wedge dy
\oplus \omega$ and denote by $\pi:\tilde{M}\rightarrow \C$ the projection.

\begin{definition}
A Lagrangian cobordism $(W; L, L')$ is a is a non-compact embedded Lagrangian $W\subset \tilde{M}$ such that for some $\epsilon > 0$ we have $\pi^{-1}([\epsilon, 1-\epsilon]\times \R) \cap W = \hat{W}$ is a smooth-compact cobordism between $L$ and $L'$ and $W\setminus \hat{W}= (-\infty,\epsilon)\times \{0\} \times L \bigcup  (1-\epsilon,
  \infty) \times \{0\} \times L'$.
\end{definition}

The relation of Lagrangian cobordism between Lagrangian submanifolds was
introduced by Arnold in \cite{Ar1}, \cite{Ar2}.  Immersed Lagrangian
cobordism was studied in \cite{Eli} and \cite{Aud} independently. Monotone embedded Lagrangian cobordism was studied in \cite{Che} and monotone embedded
Lagrangian cobordism between more than two Lagrangians was studied in \cite{BiCo},
\cite{BiCo1}. The main examples of such cobordisms are given by the
Lagrangian suspension construction \cite{Pol1} and by the trace of
Lagrangian surgery \cite{LaSik},\cite{Pol},\cite{Ha}, \cite{MaWu}.

When a Lagrangian cobordism $(W; L, L')$ is exact or monotone, what we call {\em rigid}, its
topology is very restricted. For example in \cite[Theorem 2.2.2]{BiCo} the authors proved that under natural conditions a monotone Lagrangian cobordism $(W; L, L')$ is a quantum h-cobordism. This means that $QH(W,L;\Z_{2})= 0 = QH(W,L';\Z_{2})$. Moreover, they proved that if $L, L'$ are wide\begin{footnote}{
A Lagrangian $L$ is wide (see \cite[Definition 1.2.1]{BiCo4}) if there exist an isomorphism $QH(L)\cong H_{*}(L; \Z_{2})\otimes \Z_{2}[t^{-1}, t]$, between the quantum homology and the singular homology of $L$ .}
\end{footnote} then the maps in singular homology $H_{1}(L;\Z_{2})\rightarrow H_{1}(W;\Z_{2})$ and $H_{1}(L';\Z_{2})\rightarrow H_{1}(W;\Z_{2})$ induced by the inclusions have the same image.

In this paper we investigate the maps induced by the inclusions
$L,L'\hookrightarrow W$ on the fundamental groups, $\pi_{1}(L)\xrightarrow{i_{\sharp}} \pi_{1}(W)$ and
$\pi_{1}(L')\xrightarrow{i_{\sharp}} \pi_{1}(W)$, and we prove
the following:

\begin{theorem}\label{teo:surjectivity}
  Let $(W;L,L')$ be a weakly exact Lagrangian cobordism or a
  monotone Lagrangian cobordism with $N_{W}> \text{dim}(W)$+1, then the inclusions
  $L, L'\hookrightarrow W$ induce surjective maps
  $\pi_{1}(L)\xrightarrow{i_{\sharp}} \pi_{1}(W)$ and
  $\pi_{1}(L')\xrightarrow{i_{\sharp}} \pi_{1}(W)$ on the fundamental
  groups. 
\end{theorem}

\begin{remark}
  One can expect these maps to be also injective, at least in many cases.
  Unfortunately, a better understanding of the relations in the
  fundamental group is still needed to answer this question. 
\end{remark}

\begin{remark}
  Notice that we use the construction of the fundamental group to obtain these
  results, while a somewhat more ``algebraic'' proof was pointed out to us by Baptiste Chantraine \cite{ChDiGhiGo}, using local coefficients Floer homology.  In a few words, if $i_{\sharp}$ is not onto, then in the universal cover $\tilde{W}$ of $W$, the preimage of $L$ is not connected, which is detected by the fact that
  $H_{0}(L,\Z[\pi_{1}(W)])\neq\Z$, but for an exact Lagrangian, the second author proved in \cite{Sua} that $H_{0}(L,\Z[\pi_{1}(W)])\cong H_{0}(W,\Z[\pi_{1}(W)])=\Z$.
  
As a consequence, the construction of the Floer fundamental group is not intrinsically required to prove the surjectivity of these maps.
However, we believe that the construction of the fundamental group in both the Lagrangian and Lagrangian cobordism settings has some interest in itself, as well as this version of the proof~: being slightly more geometric, this proof may also offer an interesting point of view to tackle the injectivity question.
\end{remark}

\medskip

A consequence of Theorem \ref{teo:surjectivity} is:
\begin{corollary}\label{cor:IsoIfPi2ML=0}
  If moreover $\pi_{2}(M,L)=0=\pi_{2}(M,L')$, then the maps
  $\pi_{1}(L)\xrightarrow{i_{\sharp}} \pi_{1}(W)$ and $\pi_{1}(L')\xrightarrow{i_{\sharp}} \pi_{1}(W)$ are isomorphisms. 
\end{corollary}
In \cite{Sua} the second author studied exact Lagrangian cobordism and
she showed: 
\begin{theorem} \cite[Theorem 18]{Sua}\label{teo:pseudoisotopy}
  Let $(W; L, L')$ be an exact, spin Lagrangian cobordism with
  $\mu\vert_{\pi_{2}(\tilde{M}, W)}=0$. If $L,L' \hookrightarrow
  W$ induce injective morphisms $\pi_{1}(L)\xrightarrow{i_{\sharp}} \pi_{1}(W)$
  and $\pi_{1}(L')\xrightarrow{i_{\sharp}} \pi_{1}(W)$, then $(W; L, L')$ is
  an h-cobordism. Moreover, if $dim(W) \geq 6$ then there is a diffeomorphisms $W\cong \R \times L$.
\end{theorem}
Combining Corollary \ref{cor:IsoIfPi2ML=0} with the Theorem
\ref{teo:pseudoisotopy} we have the following corollary:

\begin{corollary}\label{cor:TrivialIfPi2ML=0}
  Let $(W;L,L')$ be an exact Lagrangian cobordism such that
  $\pi_{2}(M,L)=0=\pi_{2}(M,L')$. If $W$ is spin then it is an h-cobordism.
  Moreover, if $dim(W) \geq 6$ then there is a diffeomorphisms $W\cong \R \times L$.
\end{corollary}
\begin{remark} 
  In particular, if $(W; L, L') \subset \tilde{M}$ where $M = 
  T^{*}N$ is a cotangent bundle of a closed manifold $N$, then $\pi_{2}(
  T^{*}N,  L) = 0 = \pi_{2}(
  T^{*}N,  L')$, and then Corollary \ref{cor:TrivialIfPi2ML=0}
  applies. To see this, notice that by \cite{Abu},\cite{Kr} the map
  $\rho:L\hookrightarrow T^{*}N \rightarrow N$ is a homotopy
  equivalence, on the other hand the projection $T^{*}N\rightarrow N $
  is a strong deformation retract, therefore the inclusion
  $L\hookrightarrow T^{*}N$ is a homotopy equivalence. In the same way one see $L'\hookrightarrow T^{*}N$ is a homotopy equivalence.
\end{remark}

Using similar techniques to the ones in the proof of the first statement of Theorem \ref{teo:pseudoisotopy}, based on the Biran-Cornea machinery for Lagrangian cobordisms \cite{BiCo}, we can show:

\begin{theorem}\label{theo:HcovForWeakExactOrMonotone=0}
Let $(W; L,L')$ be a weakly exact or a monotone (with $N_{W} > \text{dim}(W)+1$) spin Lagrangian cobordism. If $L,L' \hookrightarrow W$ induce injective morphisms $\pi_{1}(L)\xrightarrow{i_{\sharp}} \pi_{1}(W)$
  and $\pi_{1}(L')\xrightarrow{i_{\sharp}}\pi_{1}(W)$, then $W$ is an h-cobordism.    
\end{theorem}
Recall that an h-cobordism is a cobordism $(W; L, L')$ where the maps defined by the inclusions $L, L' \hookrightarrow W$ are homotopy equivalences.
\begin{remark}The previous theorem was proved in \cite[Corollary 2]{Sua} for exact Lagrangian cobordisms and it is based on a generalization of the result \cite[Theorem 2.2.2]{BiCo}. In fact, we expect to prove that a weakly exact or a monotone (with $N_{W} > \text{dim}(W)+1$) spin Lagrangian cobordism $(W; L,L')$ with $\text{dim}(W)\geq 6$ is deffeomorphic to $\R \times L$. For general monotone Lagrangian cobordism this is not true, Haug \cite[Theorem 1.5]{Ha} constructed examples of monotone Lagrangian cobordisms with $N_{W} < \text{dim}(W)$ and that are not diffeomorphic to  $\R \times L$. 
\end{remark}
The structure of the paper is as follows. In Section 2 we present the
definition of the Floer fundamental group of a compact Lagrangian,
adapting the construction in \cite{Pi1} to the Lagrangian setting. In
Section 3 we adapt the Floer fundamental group to a non-compact
Lagrangian with cylindrical ends. In Section 4 we give the proofs of
Theorem \ref{teo:surjectivity} and Corollary \ref{cor:IsoIfPi2ML=0} and Theorem \ref{theo:HcovForWeakExactOrMonotone=0}.

\section{Fundamental group of a Lagrangian submanifolds}
  
\subsection{Quick review of moduli spaces}\label{sec:ModuliSpaces}
Let $(M,\omega)$ be a tame symplectic manifold, this means that there is
an almost complex structure $J$ on $M$ such that $g(\cdot, \cdot) =
\omega(\cdot , J\cdot)$ is a Riemannian metric and such that the
Riemannian manifold $(M, g)$ is complete, the sectional curvature of $g$
is bounded and the injectivity radius bounded away from zero. The space
of almost complex structures on $M$ is denoted by $\mathcal{J}(M)$. Let
$J\in \mathcal{J}(M)$ be a $\omega$-compatible almost complex structure.

Let $L\subset M$ be a compact Lagrangian submanifold. There are two
morphisms associated to $L$ the symplectic area $\omega$ and the Maslov
index $\mu$: 

$$
\omega: \pi_{2}(M, L) \rightarrow \R, 
\lambda = [u] \mapsto \omega([u])= \int \limits_{D^{2}} u^{*}(\omega) 
\text{ and }\mu:\pi_{2}(M,L)\rightarrow \Z.
$$
The positive generator of the image of the homomorphism $\mu$ is called
the minimal Maslov number of $L$ and is denoted by $N_{L}$.

Lagrangians are assumed to be either 
\begin{itemize}
\item
  weakly exact: $\omega\vert_{\pi_{2}(M,L)} = 0$ or
\item 
  monotone: there is a
  positive real number $\rho \in \R$ such that for every disk $u:(D^{2},S^{1})\rightarrow (M, L)$ we have $\omega(u)= \rho\mu(u) $. 
\end{itemize} 
We assume here that all monotone Lagrangians satisfy $N_{L}> \text{dim}(L)$+1.

Let $H: M\times [0, 1] \rightarrow \R$ be a Hamiltonian, denote by
$(\phi^{t})_{t\in[0,1]}$ the induced Hamiltonian isotopy and let
$L_{1}=\phi^{1}(L)$ be such that $L\pitchfork L_{1}$.

An intersection point $x\in L\cap L_{1}$ is called \textit{fillable} if the
Hamiltonian trajectory that ends at $x$, considered as a path from $L$ to
$L$, bounds a disc $\sigma$ relative to $L$, i.e $[(\phi^{t})^{-1}(x)]=
0\in \pi_{1}(M,L)$. Such fillings have a well defined Maslov index $\mu(\sigma)$.

Let $\Orbits(L,L_{1})=\{x\in L\cap L_{1} \text{ , }[(\phi^{t})^{-1}(x)]= 0\in
\pi_{1}(M,L)\}$ be the set of \textit{fillable} intersection points. We
denote by $\CovOrbits(L,L_{1})$ the covering space with fiber $\pi_{2}(M,L)/\ker\mu$. For
$x=(\underline{x},\sigma)\in\CovOrbits(L,L_{1})$, we write $$|x|=|(\underline{x},\sigma)| = \mu(\sigma).
$$
Finally, we let $\CovOrbits_{k}(L,L_{1})=\{x\in\CovOrbits(L,L_{1}), |x|=k\}$.

\begin{remark}
Several versions of Floer trajectories can be used to define
the differential of the Floer complex. The closest one to the point of
view used by the first author in \cite{Pi1}, is that of half tubes with
boundary on $L$ satisfying the non-homogeneous Floer equation. From this
point of view, the absolute and relative cases are strictly parallel~:
the objects of the latter are all ``halves'' of the objects of the former
with boundary in $L$ and satisfy the same equations. All the arguments
are the same, as long as no bubbling occurs on $0$ and $1$ dimensional
moduli spaces.

However, the most convenient point of view for a cobordism oriented
purpose like ours is that of strips satisfying ``moving boundary
conditions''. All the Floer trajectories, augmentations or
co-augmentations are everywhere holomorphic and hence much easier to control in the cobordism framework. In consequence, the construction of
the fundamental group in the relative case will be developed using this
point of view. Aside the definition of the moduli spaces, this of course
does not essentially affect the construction which remains parallel to
the absolute case.
\end{remark}

Let $J$ be a $\omega$-compatible almost complex structure. We are
interested in the Floer moduli spaces and their classical variants. 

We fix once for all a smooth decreasing function $\beta:\R\to\R$ such that 
\begin{align*}
  \beta(s)=1 & \text{ for }s\leq-1 \\
  \beta(s)=0 & \text{ for }s\geq0 
\end{align*}
and use it to move the boundary condition on the map $u \in C^{\infty}(\R\times[0,1], M)$ by considering the equations~:
\begin{equation}
  \tag{$F_{i}$}
  \label{eq:FloerEqn_i}
  \begin{gathered}
      \forall s\in\R\text{, } u(s,0)\in L
  \\
  \forall s\in\R\text{, } u(s,1)\in \phi^{\chi_{i}(s)}(L)
  \\
  \frac{\partial u}{\partial s} + J(u)\frac{\partial u}{\partial t}=0.
  \end{gathered}
\end{equation}
for the following collection $\chi_{1},\dots,\chi_{4}$ of cut-off functions derived from $\beta$~:
\begin{enumerate}
\item 
  $\chi_{1}\equiv 1$, which defines the fixed boundary Floer equation, 
\item 
  $\chi_{2}(s)=\beta(s)$, which defines the augmentation equation, 
\item 
  $\chi_{3}(s)=\beta(-s)$, which defines the co-augmentation equation.
\item 
  $\chi_{4,R}(s)=\beta(s-R)\beta(-s-R)$, which defines the $R$-perturbed
  strip equation, where $R\in[0,+\infty)$.
\end{enumerate}
The energy of a map $u$ is defined by:
\begin{equation}
 E(u):= \iint \limits_{\R
   \times [0, 1]}\Vert\frac{\partial u}{\partial s}\Vert^{2}dsdt. 
\end{equation}
Finite energy solutions of \eqref{eq:FloerEqn_i} have converging ends,
either to a point of $L$ when $\chi_{i}(s)=0$ on this end or to an
intersection point of $L$ and $L_{1}$ when $\chi_{i}(s)=1$. Let $\U = \{u\in
C^{\infty}(\R \times [0, 1], M) \text{ }|\text{ } E(u) < \infty \}$. Let
$\star$ be a point in $L$. We are interested in the moduli spaces below~:
\begin{align*}
\M(y,x)&=
\{u\in\U,(F_{1}), \lim_{s\to\pm\infty}
u(s,t)=|\begin{smallmatrix}x\\y\end{smallmatrix},[y\sharp u\sharp \bar{x}]=0\}\\ 
\M(x,\emptyset)&=
\{u\in\U,(F_{2}), \lim_{s\to-\infty} u(s,t)=x,[x\sharp
u]=0\}\\
\M(\star,x)&=
\{u\in\U,(F_{3}), \lim_{s\to\pm\infty}
u(s,t)=|\begin{smallmatrix}x\\ \star \end{smallmatrix},[u\sharp \bar{x}]=0\}\\
\M(\star,\emptyset)&=\{(u, R)\text{, } u\in\U,\exists R\geq0 \text{ }(F_{4,R}), \lim_{s\to-\infty}
u(s,t)=\star,[u]=0\}
\end{align*}
The bracket denotes a class in $\pi_{2}(M,L)/\ker \mu$, the vanishing
conditions ensure the compatibility of the homotopy class of $u$ with
the capping of its ends $x$ or $y$ in the usual way. The overline $\bar{x}$ denotes the inverse class of the class $x=(\underline{x},\sigma)$; i.e where the capping $\sigma$ has the reversed orientation.


For a generic choice of $(H,J,\star)$, these moduli spaces are smooth
manifolds \cite{FlHoSa}. The last moduli space $\M(\star,\emptyset)$
involves constant maps for $R = 0$, for which the key argument in the
proof of transversality, of being “somewhere injective” fails.
\begin{proposition} \label{prop: transversality}
  Let $\pi:\M(\star,\emptyset)\to \R$ denote the projection. For $R = 0$,
  $\pi^{-1}(R)$ consists in the single point $(u_{\star}, 0)$ where
  $u_{\star}$ is the constant map at $\star$. Moreover, this solution is
  regular, which means that (in the suitable functional spaces) the
  equation defining the moduli space $\M(\star,\emptyset)$ is a
  submersion at this point.
\end{proposition}
\begin{proof}
  For a map $(u,R) \in \M(\star,\emptyset)$ the map
  $(\phi^{\chi_{4,R}(s)t})^{-1}(u(s,t))= v(s,t)$ has fixed boundary
  conditions on $L$ and satisfies a perturbed Cauchy-Riemann equation.
  Notice that we can reformulate the problem in terms of maps from
  $u:(\bar{D}, \partial \bar{D}) \to (M, L)$ in the trivial homology
  class, where $\bar{D}$ denotes the unit disk and $\partial \bar{D} =
  S^{1}$. The condition $\lim\limits_{s\to-\infty}u(s,t)=\star$ is then
  replaced by $u(i) = \star$.

  For $R = 0$ we have $(\phi^{\chi_{4,R}(s)t})^{-1}\equiv Id$ and
  equation \textit{(F4,R)} is the Cauchy Riemann equation: 
\begin{equation}\label{eq:Jhol}
  D(u)=\frac{\partial u}{\partial s} + J(u)\frac{\partial u}{\partial
    t}=0.
\end{equation}
Points of $\M(\star, \emptyset)$ lying above $R = 0$ are hence
$J$-holomorphic disks (with boundary condition on $\R^{n}$) in the
trivial homology class and are therefore constant. The additional
condition $u(i) = \star$ implies $\pi^{-1}(0) = \star$. The linearization
(with respect to u) of the left hand term in \ref{eq:Jhol} at the
constant map $u_{\star}$ leads to a linear operator $F$  with real
boundary condition, defined for maps from $(\bar{D}, \partial \bar{D})$
to a fixed $(\C^{n},\R^{n})= (T_{\star}M, T_{\star}L)$ of the form
\begin{equation}
  F\dot{u}=D\dot{u}+ J D \dot{u} i
\end{equation}
where $J = J_{\star}$ is constant. The Kernel of $F$ consists of the
holomorphic disks in $\C^{n}$ with boundary condition in $\R^{n}$ and
hence of the constants. It is therefore $n$-dimensional and since $n$ is
also the index of $F$, this implies that $F$ is surjective, which easily
implies the required submersion property. 
\end{proof}

The dimension depends on the ends in the following way~:
\begin{align*}
  \dim\M(y,x)&=|y|-|x|\\
  \dim\M(x,\emptyset)&=|x|\\
  \dim\M(\star,x)&=-|x|\\
  \dim\M(\star,\emptyset)&=1\\
\end{align*}
These spaces are compact up to breaks at intermediate intersection points
or bubbling of discs. However, the condition $\omega_{|\pi_{2}(M,L)}=0$
guaranties that there are no non-constant holomorphic discs in $M$ with
boundary on $L$. For monotone Lagrangian, bubbling may occur, either
``on the side'' of the strip, or the strip itself might brake in the area
where the hamiltonian term is turned off. The condition $N_{L}\geq 3$ is
enough to prevent side bubbling on moduli spaces up to dimension $2$, but
the second phenomenon may still occur on the moduli space
$\M(\star,\emptyset)$, leading to configurations that consist of an
holomorphic disc through $\star$ in some homotopy class $\alpha$,
followed by a Floer strip where the hamiltonian term is turned off near
both ends in the homotopy class $-\alpha$. The existence of such
configurations is forbidden by the condition $N_{L}> \text{dim}(L)+1$ .

All the moduli spaces of interest to us are hence compact up to breaks at
intermediate intersection points, and there is a gluing construction
ensuring that each broken configuration does indeed appear on the
boundary of some moduli space.

In particular, given some $x\in\CovOrbits_{0}(L,L_{1})$ and
$y\in\CovOrbits_{1}(L,L_{1})$, each broken trajectory
$(\beta,\alpha)\in\M(y,x)\times\M(x,\emptyset)$ belongs to the boundary
of one component of $\M(y,\emptyset)$. Since this last space is
$1$-dimensional, it has to have another end
$(\beta',\alpha')\in\M(y,x')\times\M(x',\emptyset)$ for some
$x'\in\CovOrbits_{0}(L,L_{1})$. We denote this relation between
$(\beta,\alpha)$ and $(\beta',\alpha')$ by
$$
(\beta,\alpha)\glueto(\beta',\alpha').
$$
\begin{definition}\label{def: augmentation}
  Given an index 0 intersection point $x \in \CovOrbits_{0}(L,L_{1})$, a
  capping $\alpha \in \M(x,\emptyset)$ is called an “augmentation” of
  $x$, and the couple $(x, \alpha)$ an augmented orbit.
\end{definition}

\subsection{Steps and loops}

\begin{definition}\label{def:FloerLoop}
  A Floer step is an oriented connected component with non-empty boundary
  of a 1-dimensional moduli space $\M(y,\emptyset)$ for $y\in\CovOrbits_{1}(L,L_{1})$ or
  $\M(\star,\emptyset)$.
\end{definition}

More explicitly, a Floer step in $\M(y,\emptyset)$ for
$y\in\CovOrbits_{1}(L,L_{1})$ can be identified to a quadruple
$(\alpha_{0},\beta_{0},\beta_{1},\alpha_{1})$ where
$\beta_{i}\in\M(y,x_{i})$ for some $x_{i}\in\CovOrbits_{0}(L,L_{1})$, and
$\alpha_{i}\in\M(x_{i},\emptyset)$ are such that
$$
(\beta_{0},\alpha_{0})\glueto(\beta_{1},\alpha_{1}).
$$

In $\M(\star,\emptyset)$, exactly one special step
$$
(\star) \glueto (\alpha_{\star},\beta_{\star})
$$ 
has the constant $J$-holomoprhic disc at $\star$ as one end (and a broken
configuration $(\beta_{\star},\alpha_{\star})$ as the other end), while
all the other steps can be described as a quadruple
$(\alpha_{0},\beta_{0},\beta_{1},\alpha_{1})$ where
$\beta_{i}\in\M(\star,x_{i})$ for some $x_{i}\in\CovOrbits_{0}(L,L_{1})$, and
$\alpha_{i}\in\M(x_{i},\emptyset)$ are such that
$$
(\beta_{0},\alpha_{0})\glueto(\beta_{1},\alpha_{1}).
$$

With these notations, $\alpha_{0}$ is the start of
the step, and $\alpha_{1}$ its end.

\begin{definition}
A Floer based loop is a sequence of consecutive Floer loop steps starting and ending at $\star$.

In other words a Floer based loop is a sequence of consecutive steps that starts and ends at $\star$, i.e. a sequence
$(\star,\alpha_{0},\beta_{0})(\alpha_{0},\beta_{0},\beta'_{0},\alpha'_{0}),\dots
(\alpha_{N},\beta_{N},\beta'_{N},\alpha'_{N})(\alpha'_{N},\beta'_{N},\star)$
such that~:
  $$
  \alpha_{0}=\alpha_{\star},\ 
  \beta_{0}=\beta_{\star},\ 
  \alpha'_{0}=\alpha_{1},\ \dots,\ 
  \alpha'_{N-1}=\alpha_{N},\
  \alpha'_{N}=\alpha_{\star},\
  \beta'_{N}=\beta_{\star}.
  $$
and
  $$
  \forall i, (\beta_{i},\alpha_{i})\glueto(\beta'_{i},\alpha'_{i})
  $$
  The set of all Floer loops is denoted by $\tilde{\Gen}(L,L_{1})$.
\end{definition}

Notice that $\tilde{\Gen}(L,L_{1})$ depends on all the auxiliary data $(H,
\star, J,\chi)$ but the dependency on the almost complex structure $J$
and the cut off function $\chi$ is kept implicit to reduce the notation.
$\tilde{\Gen}(L,L_{1})$ carries an obvious concatenation rule that turns it
into a semi-group. It also carries obvious cancellation rules. More
explicitly, if $\sigma = (\alpha, \beta, \beta ' , \alpha ')$ is a Floer
loop step, define its inverse $\sigma^{-1}$ to be the same step with the
opposite orientation: $\sigma^{-1} = (\alpha ', \beta ', \beta, \alpha)$.

Denote by $\sim$ the associated cancellation rules in $\tilde{\Gen}(L,L_{1})$.
The concatenation then endows the quotient space
$$\Gen(L,L_{1})= \tilde{\Gen}(L,L_{1})\diagup \sim$$
with a group structure.

Picking a parametrisation by $[0,1]$ of all the relevant components of
the different moduli spaces, it also carries an evaluation map to the
based loop space of $L$. More precisely, recall that each strip $u$ in
$\M(y,\emptyset)$ or $\M(\star,\emptyset)$ converges to a point in $L$ as
$s$ goes to $+\infty$ denoted by $u(+\infty)$~; given a parametrisation
$[0,1]\to\M(y,\emptyset)$ of a component $\M$ of such a moduli space, the
map
$$
\begin{array}{ccc}
\M& \to & L\\
u &\mapsto &u(+\infty)  
\end{array}
$$
defines a path in $L$. Concatenating the paths associated to all the
loops, we get a loop in $L$ based at $\ast=\alpha_{\star}(+\infty)$~:
$$
\Gen(L,L_{1})\xrightarrow{\eval}\pi_{1}(L,\ast).
$$

The main statement of this section is then the following~:
\begin{theorem}\label{thm:Onto}
  If $L$ is weakly exact or monotone with $N_{L}> \text{dim}(L) + 1$, then the evaluation
  map
  $$
  \Gen(L,L_{1})\xrightarrow{\eval}\pi_{1}(L,\ast)
  $$
  is onto.
\end{theorem}

In other words, if $\sim$ is the homotopy equivalence relation in
$\Gen(L,L_{1})$, then~:
\begin{corollary}
  The map $\Gen(L,L_{1})/_{\sim} \xrightarrow{\eval}\pi_{1}(L,\ast)$ is an
  isomorphism.
\end{corollary}
For a more economical presentation of the relations, but that uses an
auxiliary Morse function, we refer to \cite{Pi1} where the proposed
description should adapt straightforwardly.

The idea of the proof of theorem \ref{thm:Onto} is to associate, to each
Morse loop, an homotopic Floer loop. Since it is well known that Morse
loops generate the fundamental group, this is enough to obtain theorem \ref{thm:Onto}. 

\subsection{From Morse to Floer loops}\label{sec:FloerToMorse}

Pick a Morse function $f$ on $L$ with a single minimum at $\star$ and a
Riemannian metric $g$ on $L$. Assume $f, g$ are both chosen generically
so that all the Morse and PSS moduli spaces are defined transversely.
Namely, we are interested in the following moduli spaces~:
\begin{align*}
\M(b,a)&=(W^{u}(b)\cap W^{s}(a))/\R,
\text{ for } a,b\in\Crit(f),
\\ 
\M(b,x)&=\{u\in\M(\emptyset,x), u(-\infty)\in W^{u}(b)\},
 \text{ for } b\in\Crit(f), x\in\CovOrbits(L,L_{1}),
\\ 
\M(y,a)&=\{u\in\M(y,\emptyset), u(+\infty)\in W^{s}(a)\},
 \text{ for } y\in\CovOrbits(L,L_{1}), a\in\Crit(f).
\end{align*}

A Morse step is an oriented loop around the unstable manifold of an index
$1$ critical point of $f$. The space of Morse loops
$$
\Gen(f)=<\Crit_{1}(f)>
$$
is then the free group generated by the index $1$ critical points of $f$. 
Its elements encode sequences of oriented loops around the unstable
manifolds of index one critical points, and hence have a realization as
loops in $L$~:
$$
\Gen(f)\xrightarrow{\eval}\Loops(L,\star)\xrightarrow{\pi}\pi_{1}(L,\star)
$$

%

Let $b$ be an index $1$ critical point of $f$. Denote by $\gamma_{\pm}$ the
two Morse flow lines rooted at $b$~:
$$
\M(b,\star)=\{\gamma_{-},\gamma_{+}\}.
$$
Consider the space
$$
B_{b}=\bigcup_{\substack{y\in\CovOrbits_{1}(L,L_{1})\cup\{\star\}\\x\in\CovOrbits_{0}(L,L_{1})}}
\M(b,y)\times\M(y,x)\times\M(x,\emptyset)
\bigcup
\{(\gamma_{-},\star),(\gamma_{+},\star)\}
$$ 
of twice broken or degenerate hybrid Morse and Floer trajectories from $b$ to
$\emptyset$.

The ``upper gluing'' map $\sharp^{\bullet}$ is defined as
$$
\sharp^{\bullet}:
\begin{array}{ccc}
B_{b} & \to & B_{b}\\
(\gamma,\beta,\alpha) & \mapsto &(\gamma',\beta',\alpha)
\rlap{ such that $(\gamma,\beta)\glueto(\gamma',\beta')$}\\
(\gamma_{\pm},\star)& \mapsto & (\gamma_{\mp},\star)
\end{array}
$$

The ``lower gluing'' map $\sharp_{\bullet}$ is defined as
$$
\sharp_{\bullet}:
\begin{array}{ccc}
B_{b} & \to & B_{b}\\
(\gamma,\beta,\alpha) & \mapsto &(\gamma,\beta',\alpha')
\rlap{ such that $(\beta,\alpha)\glueto(\beta',\alpha')$}\\
(\gamma_{\pm},\star)& \mapsto & (\gamma_{\pm},\beta_{\star},\alpha_{\star})
\end{array}.
$$

Both $\sharp^{\bullet}$ and $\sharp_{\bullet}$ are one to one maps. We
refer to alternating iteration of $\sharp^{\bullet}$ and
$\sharp_{\bullet}$ as running a ``crocodile walk'' on $B_{b}$.

Since $B_{b}$ is finite, all the orbits of the crocodile walk are
periodic. Notice that the two degenerate configurations
$(\gamma_{-},\star)$ and $(\gamma_{+},\star)$ are in the same orbit.

%
%
%

Let 
$$%
\xymatrix{
(\gamma_{-},\star)\ar^{\sharp_{\bullet}}[r]&
(\gamma_{0}, \beta_{0}, \alpha_{0})\ar^{\sharp^{\bullet}}[r]&
\dots \ar^{\sharp^{\bullet}}[r]&
(\gamma_{N}, \beta_{N}, \alpha_{N})\ar^{\sharp_{\bullet}}[r]&
(\gamma_{+},\star)\ar^{\sharp^{\bullet}}@/^1pc/[llll]
}
$$
be the orbit containing the degenerate configurations.

The lower parts of these broken trajectories form a Floer loop
$$
\psi(b)=\big(\star,(\beta_{0},\alpha_{0}), \dots,
(\beta_{N},\alpha_{N}),\star\big).
$$
Extending this map by concatenation, we get a group morphism
$$
\Gen(f)\xrightarrow{\psi}\Gen(L,L_{1}).
$$
Consider now the following diagram
\begin{equation}
  \label{eq:evalDiagramPsi}
  \xymatrix{
  \Gen(f) \ar[r]^{\eval}\ar[d]^{\psi} 
& \pi_{1}(L,\star)\ar[d]^{\Id}\\
  \Gen(L,L_{1}) \ar[r]^{\eval} 
& \pi_{1}(L,\ast)\\
}.
\end{equation}

\begin{proposition}\label{prop:comutative}
  The diagram \ref{eq:evalDiagramPsi} is commutative, i.e. for all
  $\gamma\in\Gen(f)$, $\gamma$ and $\psi(\gamma)$ are homotopic.
\end{proposition}
\begin{proof}
  Let $b$ be an index $1$ Morse critical point and consider the orbit of the crocodile walk introduced above to define $\psi(b)$.

  Consider a trajectory $(u,v)$ in a space of the form
  $\M(b,y)\times\M(y,\emptyset)$ for some
  $y\in\CovOrbits_{1}(L,L_{1})\cup\{\star\}$ (or
  $\M(b,x)\times\M(x,\emptyset)$ for some $x\in\CovOrbits_{0}(L,L_{1})$).

  Evaluation of $u$ and $v$ along $\R\times\{0\}$ turns them into paths
  defined on $\R$. Parameterising their Morse part by the value of $f$ and
  the Floer parts by the (suitable version of) the action turns them into Moore
  paths, and after concatenation, we get a Moore path in $L$ rooted at
  $b$. This defines a map
  $$
  \M(b,y)\times\M(y,\emptyset)\to\mathcal{P}(L,b)
  $$
  where $\mathcal{P}(L,b)$ is the space of Moore paths in $L$ rooted at
  $b$. Notice this map is compatible with the compactification of the
  moduli spaces and extends continuously to their boundary.

  The same holds for spaces of the form $\M(b,x)\times\M(x,\emptyset)$
  for $x\in\CovOrbits_{0}(L,L_{1})$ or the Morse space $\M(b,\emptyset)$, and
  hence, each upper or lower gluing in the crocodile walk defines a
  continuous one-parameter family of paths rooted at $b$, and whose end
  follows the evaluation of the step associated to that gluing. Moreover,
  along the orbit of the crocodile walk, each such family ends where the
  next one starts, so that they globally form a continuous $\S^{1}$
  family of rays, all rooted at $b$ and whose other end describes the
  loop associated to that orbit. 

  This means that the orbit of the crocodile walk defines a trivial loop.
  Notice that one of the steps in this orbit is nothing but
  $\gamma_{+}\glueto\gamma_{-}$, whose evaluation is  the Morse step
  associated to $b$, while the evaluation of the remaining part is
  $\eval(\psi(b))$. As consequence, $\eval\psi(b)$ is homotopic to the
  Morse step associated to $b$.

  Repeating this for each index $1$ Morse critical point $b$, we obtain
  the result.
 \end{proof}

\subsubsection{Proof of theorem \ref{thm:Onto}}

Theorem \ref{thm:Onto} is a straightforward corollary of proposition
\ref{prop:comutative}. Recall $\Gen(f) \xrightarrow{\eval}
\pi_{1}(L,\star)$ is onto. To see this, one can pick a representative
$\gamma$ of an homotopy class $[\gamma]\in\pi_{1}(L,\star)$. Using a
genericity condition, one can suppose that $\gamma$ does not meet any
stable manifold of codimension at most $2$, and meets the stable
manifolds of the index $1$ critical points transversely. Pushing
$\gamma$ down by the gradient flow of the Morse function then defines a
homotopy from $\gamma$ to a Morse loop. 

\begin{remark}
  When $\gamma$ is a Floer loop, this process can be interpreted in terms
  of suitable moduli spaces and gluings, and fits exactly in the same
  formalism as the one used to move a Morse loop into a Floer loop.  
\end{remark}

Since $\Gen(f)\xrightarrow{\eval}\pi_{1}(L,\star)$ is onto and the
diagram \eqref{eq:evalDiagramPsi} commutative,
$\Gen(L,L_{1})\xrightarrow{\eval}\pi_{1}(L,\star)$ has to be onto as well.

\section{Fundamental group of Lagrangians with cylindrical ends}

We now turn to Lagrangian submanifolds with cylindrical ends in
$(\C\times M,\omegastd\oplus\omega)$ as defined in \cite{BiCo}.

Let $W$ be a Lagrangian submanifold in $\C\times M$ with $N_{-}$
negative cylindrical ends $(L^{-}_{1},\dots,L^{-}_{N_{-}})$, and $N_{+}$
positive ones $(L^{+}_{1},\dots,L+_{N_{+}})$~: this means that there
exist some real number $A$ and collections of distinct constants
$a^{\pm}_{1},\dots,a^{\pm}_{N_{\pm}}$, such that 
\begin{gather*}
  E_{\pm}(A):=W\cap\pi^{-1}([\pm A,\pm\infty)\times \R) =
 \bigsqcup_{i=1}^{N_{\pm}}\Big([\pm A,\pm\infty)\times\{a^{\pm}_{i}\}\Big)\times L^{\pm}_{i}
\intertext{and}
W\cap\pi^{-1}([-A,A]) \text{ is compact}.
\end{gather*}

The submanifold $W$ is supposed to satisfy
$\tilde{\omega}(\pi_{2}(\tilde{M},W))=0$ or to be monotone with minimal Maslov number $N_{W}> \text{dim}(W) + 1$.

The previous construction requires some adaptation to make sens in this non compact setting, and following \cite{BiCo}, we consider~:
\begin{itemize}
\item 
 a compact region $B\subset \C$ in the plane, outside which $W$ is cylindrical.
\item
 an almost complex structure $\tilde{J}$ on $\C\times M$, 
 $\tilde\omega$-compatible and such that $\tilde{J}=i\oplus J$ outside
 $B\times M$ for some $\omega$-compatible almost complex structure $J$ on
 $M$. 
\item
 A Hamiltonian function $\tilde{H}:[0,1]\times\C\times M\to\R$ that is
 supported over $B$ and disjoint neighborhoods $V^{\pm}_{i}=\{(x,y)\in\C, \pm x>
  A,|y-a^{\pm}_{i}|<\epsilon \}$ of the ends and is linear in the real coordinate over each cylindrical end~:
 $$
 \forall (x,y,m)\in V^{\pm}_{i}\times M,\quad  \tilde{H}(x,y,m)=\alpha^{\pm}_{i}x+\beta^{\pm}_{i}
 $$
\end{itemize}
 
Unlike the homology, by compactness reasons the construction of the fundamental group does not allow arbitrary perturbations of the ends, and we require
\begin{equation}
  \label{eq:EndSigns}
\begin{gathered}
  \forall i\in\{1,\dots,N_{-}\},\quad \alpha^{-}_{i}<0\\
  \forall i\in\{1,\dots,N_{+}\},\quad \alpha^{+}_{i}>0.
\end{gathered}
\end{equation}

Finally, we pick a point $\star$ in $W$.

\subsection{Moduli spaces and Floer loops}
Let $W_{1}=\phi_{\tilde{H}}(W)$ be the image of $W$ under the Hamiltonian isotopy associated to $\tilde{H}$. The set $\CovOrbits(W,W_{1})$ is defined in the same way as before, and we are interested in the moduli spaces $\M(y,x)$, $\M(x,\emptyset)$, $\M(\star,x)$, $\M(\star,\emptyset)$ introduced in section \ref{sec:ModuliSpaces}.

For a generic choice of the auxiliary data $(\tilde{H},\tilde{J},\star)$ they are smooth manifolds of the expected dimension, but due to the non-compact setting, their compactness require special attention and is discussed in details in \cite[section 5.2]{BiCo}. More precisely, in this work P.~Biran and O.~Cornea use a Morse function $f$ on $W$ with $f=\tilde{H}$ over the ends, and discuss the compactness of hybrid Floer-Morse moduli spaces $\M(x,a)$ or $\M(a,x)$ for $x\in\CovOrbits(W,W_{1})$ and $a\in\Crit(f)$
given as
\begin{align*}
\M(x,a)&=\{u\in\M(x,\emptyset), u(+\infty)\in W^{s}(a)\}\\
\M(a,x)&=\{u\in\M(\emptyset,x), u(-\infty)\in W^{u}(a)\}.
\end{align*}

The moduli spaces $\M(\star,x)$ fit into the description above (picking a Morse function $f$ that has a minimum at $\star$), we are also interested in moduli spaces of the form $\M(x,\emptyset)$ or $\M(\star,\emptyset)$, where the Morse constraint is removed. It turns out that this causes the moduli spaces not to be compact up to breaking any more if the ends are not moved in the right direction, i.e. if \eqref{eq:EndSigns} is not fulfilled. 

If $f$ is ``increasing'' at infinity (i.e. $f(x)$ is increasing with $|x|$ for $|x|$ large enough), so that its flow pushes any point down to a compact set and hence to a critical point. As a consequence the space $\M(x,\emptyset)$ can be seen as the union
$\bigcup_{a\in\Crit(f)}\M(x,a)$ and is hence compact up to breaks (bubbles are not allowed due to the usual asphericity or dimensional arguments).

As consequence, the notions of loops given in
\ref{def:FloerLoop} make sense in this non-compact setting.

\subsection{Homotopies}

We fix a Morse function $f$ (and a Riemannian metric $g$) on $W$ with a single minimum at $\star$, and such that $f=\tilde{H}$ up to a constant over the ends.

The restriction \eqref{eq:EndSigns} ensures that the flow of the Morse function pushes any loop in $W$ down to a compact set and hence to the $1$-skeleton given by $f$~: the homotopy $\phi$ defined in \ref{sec:FloerToMorse} still make sens in this non-compact setting.

Similarly, once all the relevant moduli spaces are proven to be compact up to breaking, the crocodile walk can be run just like in the compact case, and the homotopy $\psi$ defined in \ref{sec:FloerToMorse} is still well defined in this non-compact setting. The diagram \eqref{eq:evalDiagramPsi} still hold with the same commutative up to homotopy property.

Moreover, the required behavior of $f$ over the ends ensures that the Morse loops do indeed generate the fundamental group despite of the non-compactness of $W$. As a consequence, the Floer loops do also generate the fundamental group.

\begin{theorem}
  $\Gen(W,W_{1})\xrightarrow{\eval}\pi_{1}(W)$ is onto.
\end{theorem}

\section{Cobordisms with two ends}
We now restrict attention to the case when $W$ has exactly one negative end $L_{-}$ and one positive end $L_{+}$ (and still satisfies the topological restriction $\omega(\pi_{2}(\tilde{M},W))=0$ or $N_{W}\geq \text{dim}(W) + 2$).

The goal of this section is to prove Theorem \ref{teo:surjectivity}.
To this end, we use suitable auxiliary data $(\tilde{H},\tilde{J},\star)$ to compare $\pi_{1}(W)$ and $\pi_{1}(L_{-})$.

Namely, a first Hamiltonian $\tilde{H}$ is chosen so that the associated isotopy $(\Phi^{t})$ moves $W\cap\pi^{-1}([-A,A])$ far enough downward (in the $y$-direction) in the $\C$ factor to separate it from itself, while it translates the cylindrical ends vertically, upward for the negative end and downward for the positive one: for $a_{-}, a_{+}, A', K \in \R$ with $a_{-} > 0$ and $a_{+} < 0$ we have $\pi(\Phi^{1}((-\infty, A']\times \{0\}\times L_{-})) = (-\infty, A'\pm K ]\times \{a_{-}\}$ and $\pi(\Phi^{1}([A',\infty)\times \{0\}\times L_{+})) = [A' \pm K,\infty)\times \{a_{+}\}$.

Then all the intersections of $W$ and $W_{1}$ lie above a single point $p\in\C$, that belongs to the area where the projection $\pi:\C\times M\to\C$ is holomorphic and where $\pi(W)$ and $\pi(W_{1})$ are two transverse curves. These curves divide a neighborhood of $p$ in $4$ regions~: call the ones whose natural boundary orientation goes from $W$ to $W_{1}$ ``arrival regions'' and the other ``departure regions''. The isotopy $(\Phi^{t})$ can be chosen to be a product $(\phi^{t}_{\C},\phi^{t}_{M})$ near the fiber $\pi^{-1}(p)$, where $\phi^{t}_{\C}$ leaves $p$ fixed, and $\phi^{t}_{M}$ moves $L_{-}$ to some Lagrangian submanifold $L_{1}$ transverse to $L_{-}$.

In particular, the projection $\pi(\bigcup_{0\leq t\leq 1}\Phi^{t}(W))$ can be supposed to fill the arrival regions and not to touch the departure ones, and the components of $\C\setminus\pi(\bigcup_{0\leq   t\leq 1}\Phi^{t}(W))$ that contain a departure region to be unbounded.

Recall from \cite{BiCo} that 
$$
\forall y\in\CovOrbits_{1}(W,W_{1}),
\forall x\in\CovOrbits_{0}(W,W_{1}),
\forall u\in\M(y,x),\quad
\pi(u(\R\times[0,1]))=\{p\}.
$$

The same kind of arguments apply to augmentations and co-augmentations
through $\star$ if $\pi(\star)=p$, and we have the two following lemmas~:
\begin{lemma}
For all $x$ in $\CovOrbits_{0}(W,W_{1})$, we have
$$
\forall u \in \M(x,\emptyset):\quad \pi(u(\R\times[0,1]))=\{p\}.
$$
\end{lemma}

\begin{proof}
  For orientation reasons, if $\pi(u(\R\times[0,1]))$ is not constant, by the open mapping theorem the strip $\pi\circ u$
  has to meet one of the unbounded regions of
  $\C\setminus(\pi(u(\R\times\{0\}))\cup\pi(u(\R\times\{1\})))$ when leaving $p$, which is not possible. 
\end{proof}

\begin{lemma}
Suppose the point $\star$ is chosen in $\pi^{-1}(p)$, then for all $x$ in
$\CovOrbits_{0}(W,W_{1})$ and $u$ in $\M(\star,x)$, we have
$$
u(\R\times[0,1])\subset\pi^{-1}(p).
$$
\end{lemma}
\begin{proof}
  The proof is slightly different here since the strip $\pi\circ u$ now
  ends at $p$ instead of starting from it.

  More precisely, it still starts at $\pi(\star)=p$, but the boundary conditions on the $-\infty$ end are now given by $L_{-}$ on both sides.

  After a suitable conformal transformation sending the strip to
  $D=\{z\in\C,|z|\leq1\}\setminus\{-1,1\}$, $u$ defines a map $\tilde{u}$
  on the closed unit disc $\bar{D}$ with $\tilde{u}(-1)=\star$ and
  sending a neighborhood of $-1$ in $\partial \bar{D}$ to $L_{-}$. In
  particular, the image of this neighborhood has to meet at least one of  the departure regions. Since these regions are unbounded, this is a contradiction.
\end{proof}

\begin{remark}[Genericity]
  The condition $\pi(\star)=p$ is obviously not a generic one. However, it still offers enough freedom to ensure transversality, i.e. that all the relevant moduli spaces are cut out transversely. In fact, we just proved that all the trajectories we are interested are contained in the single fiber $M=\pi^{-1}(p)$. A generic choice of $\star$ (and all the   other auxiliary data) on $L_{-}$ ensures transversality for the relevant moduli spaces of curves in $M$. To ensure transversality for the same curves seen in $\C\times M$, we need to check that the linearized operators $\bar{\partial}_{u}$ (with its associated boundary conditions) at each curve $u$ seen in $\C\times M$ are onto.

But for the chosen structure, $\bar\partial_{u}$ splits as a direct sum $\bar{\partial}\oplus\bar{\partial}_{u_{M}}$ where the first $\bar{\partial}$ is associated to the constant curve $\{p\}$ and boundary conditions associated to $T_{p}\pi(W)$ and $T_{p}\pi(W_{1})$. Then the linearized operator associated to $\bar{\partial}$ is also surjective by the same argument that in Proposition \ref{prop: transversality} (its Kernel is real constant maps and then its $Index$ equal to the dimension of its Kernel). By assumption the linearization of $\bar{\partial}_{u_{M}}$ is surjective, therefore the operator corresponding to the linearization of $\bar{\partial}\oplus\bar{\partial}_{u_{M}}$ is surjective.

\end{remark}

\subsubsection{Proof of Theorem \ref{teo:surjectivity}} 

Let $L_{1}=\phi_{M}^{1}(L_{-})$. From the previous section, we derive that Floer steps in $(M,L_{-},L_{1})$ are the same as the Floer steps in $(\C\times M,
W,W_{1})$. Moreover, loops that are homotopically trivial in $L_{-}$ are
obviously also trivial in $W$, so that the identity
$\Gen(L_{-},L_{1})\xrightarrow{\Id}\Gen(W,W_{1})$ at the generator level induces a
map $\pi_{1}(L_{-},L_{1})\xrightarrow{i}\pi_{1}(W,W_{1})$ at the fundamental groups
level, which is obviously onto.

Moreover, since the identity commutes with the evaluation map, the
following diagram, where $i_{\sharp}$ is the map induced by the inclusion,
commutes~:
$$
\xymatrix{
\pi_{1}(L_{-},L_{1})\ar@{->>}[r]^{i}\ar[d]_{\eval}^{\sim} & \pi_{1}(W,W_{1})\ar[d]_{\eval}^{\sim}\\
\pi_{1}(L_{-},\ast)\ar[r]^{i_{\sharp}} & \pi_{1}(W,\ast)\\
}.
$$
This ends the proof that $i_{\sharp}$ is onto for $(W; L_{-}, L_{+})$ with $N_{W}> dim(W) +1$. 


\subsubsection{Case $\pi_{2}(M,L_{-})=0$}
Suppose moreover that $\pi_{2}(M,L_{-})=0$. From the homotopy long exact sequence of the pair $(M,L_{-})$ we derive that the inclusion of $L_{-}$ in $M$ induces an injective group morphism
$\pi_{1}(L_{-})\hookrightarrow\pi_{1}(M)$. This also holds for the
inclusion of $L_{-}$ in $\C\times M$, and since it factors
through the inclusion of $L_{-}$ in $W$, we conclude that the induced
map $\pi_{1}(L_{-})\xrightarrow{i_{\sharp}}\pi_{1}(W)$ has to be injective.

This ends the proof of corollary \ref{cor:IsoIfPi2ML=0}.

\subsubsection{Proof of Theorem \ref{theo:HcovForWeakExactOrMonotone=0} }
To prove this theorem we use a version of Floer homology called quantum homology, defined in \cite{BiCo1} for a Lagrangian cobordism $(W;L,L')$. Denote by $QH(W,L;\Z[\pi_{1}(W)])$ the homology of the complex $\mathcal{C}(f, \rho, J;Z[\pi_{1}(W)])= (\mathcal{C}_{*}(f, \rho, J), d_{*})$. This complex is the $\Z$-graded \footnote{To define the complex with $\Z$-coefficients, it is enough the Lagrangian to be spin. We use the conventions for orientations in \cite{BiCo3}, that in particular imply that the orientations for Morse complex and pearl complex coincide.} free $\Z[\pi_{1}(W)]$-chain complex defined for a triple $\mathcal{D}=(f, \rho, J)$ composed by a Morse function $f : \hat{W} \rightarrow \R$ a Riemannian metric $\rho$ on $W$ and an almost complex structure $J$. The data is assumed to be generic. We assume that the gradient vector field $\nabla_{\rho}f$ is transverse to $\partial \hat{W}$, and it  points outside along $L$ and it points inside along $L'$. The function $f$ is extended linearly to $W$. The ring $\Lambda = \Z[t^{-1}, t]$ is graded by setting $\text{deg}(t) = -N_{W}$, then  $\Lambda= \bigoplus \limits_{i\in \Z}\Lambda_{i}$, where $\Lambda_{i}$ is the subring of homogeneous elements of degree $i$. Denote by $\text{Crit}_{i}(f)$ the set of index $i$ critical points of $f$, then $$\mathcal{C}_{*}(f, \rho, J):= \bigoplus\limits_{i \in \Z} \Z[\pi_{1}(W)]\otimes \Z \langle \text{Crit}_{* + iN_{W}}(f)\rangle \otimes \Lambda_{-iN_{W}},$$ the differential $d:\mathcal{C}_{*}(f, \rho, J)\rightarrow \mathcal{C}_{*-1}(f, \rho, J)$ counts configurations in the moduli space $\mathcal{P}(x, y, A; \mathcal{D})$, where $x, y \in \text{Crit}(f)$ and $A\in \pi_{2}(\tilde{M}, W)$. An element of  $\mathcal{P}(x, y, A; \mathcal{D})$ is a pearl trajectory, this is a configuration of flow lines joining $J$-holomorphic disks. 
An element of $\mathcal{P}(x, y, A; \mathcal{D})$ can be written by an array of disks $(u_{1}, \ldots, u_{k})$ where each $u_{i}$ is a $J$-holomorphic disk such that $u_{1}(-1) \in W^{u}_{x}(f)$ the unstable submanifold of $f$ at $x$, for $0 \leq i \leq k$ the disk $u_{i}(D^{2})$ is connected to $u_{i+1}(D^{2})$ by a flow line from $u_{i}(1)$ to $u_{i+1}(-1)$, $u_{k}(1)\in W^{s}_{y}(f)$ (the stable submanifold of $f$ at $y$) and $\mu(u_{1}) + \cdots + \mu(u_{k}) = \mu(A)= lN_{W}$. 
The dimension of $\mathcal{P}(x, y, A; \mathcal{D})$ is given by $|x|-|y| + \mu(A)- 1$. For a detailed definition and for a proof of the regularity of these spaces see \cite{BiCo0}. 

For each $z \in \text{Crit}(f)$, denote by $\gamma_{z}$ a fixed oriented path connecting $z$ to a fixed base point $\star \in W$, we denote by $\bar{\gamma_{z}}$ the path $\gamma_{z}$ with reversed orientation.
For $\overline{u}=(u_{1}, \ldots, u_{k}) \in \mathcal{P}(x, y, A; \mathcal{D})$ denote by $g_{\overline{u}}$ the homotopy class of the loop obtained by concatenation of $\bar{\gamma_{x}}\# \partial_{-}\overline{u} \# \gamma_{y}$, where $\partial_{-}\overline{u}$  is the path from $x$ to $y$ obtained by following the flow lines and half boundaries of the disks $u_{i}(D^{2})$ given by $u_{i}(e^{i\theta})$ with $\pi \leq \theta \leq 2\pi$. 

$$d(x) := \sum \limits_{\substack{y\in \text{Crit}(f)\\
|x|-|y| = 1}} \sum \limits_{u\in \mathcal{P}(x, y, 0; \mathcal{D})} g_{u}y  +\sum \limits_{\substack{A \in \pi_{2}(\tilde{M},W), y\in \text{Crit}(f)\\
|x|-|y| + \mu(A) = 1}}\sum \limits_{u\in \mathcal{P}(x, y, A; \mathcal{D})} (-1)^{|y|}g_{u}y  t^{\mu(A)}.$$

The differential satisfies $d^{2} = 0$, this is a straightforward adaptation of the proof given in \cite{BiCo3}.

By grading reasons and the assumption $N_{W}> \text{dim}(W) + 1$, for $0 \leq * \leq \text{dim}(W)$ we have that $\mathcal{C}_{*}(f,\rho, J)= \Z[\pi_{1}(W)]\otimes\Z\langle\text{Crit}_{*}(f)\rangle$. For these dimensions the space $\mathcal{P}(x, y, A; \mathcal{D}) = \mathcal{P}(x, y, 0; \mathcal{D})$ consists of flow trajectories connecting $x$ to $y$:
by assumption $N_{W} > \text{dim} W$ +1, so if $\mu(A) \neq 0$ then $\mu(A) > |x|-|y| - 1$ and then $\mathcal{P}(x, y, A; \mathcal{D})= \emptyset$.
This means that $d_{*}= d_{Morse}$ is the Morse differential for $0 \leq * \leq \text{dim} W$.
So the complex $\mathcal{C}(f, \rho, J;Z[\pi_{1}(W)])$ contains a copy of the $\Z[\pi_{1}(W)]$-Morse complex of the pair $(W,L)$, let us denote it by $CM(W,L;\Z[\pi_{1}(W)])$.
\begin{align}
\ldots \rightarrow C_{\text{dim}(W)+1}(f, \rho, J) &\xrightarrow{d_{\text{dim}(W)+1}} C_{\text{dim}(W)}(f, \tilde{\rho}, J) \rightarrow \ldots  \rightarrow C_{0}(f, \rho, J)\xrightarrow{d_{0}}\ldots\\
\ldots \rightarrow C_{\text{dim}(W)+1}(f, \rho, J) &\xrightarrow{d_{\text{dim}(W)+1}} CM(W,L;\Z[\pi_{1}(W)]) \xrightarrow{d_{0}}C_{-1}(f, \rho, J)\rightarrow \ldots 
\end{align}
Recall that the $\Z[\pi_{1}(W)]$-Morse complex of the pair $(W,L)$ is the complex associated to a Morse function on $W$ with differential defined in the same way as $d$ before but taking into account only the spaces $\mathcal{P}(x, y,A; \mathcal{D})$ with $A = 0$. By lifting the pair $(f, \rho)$ to $\tilde{W}$, we can define a cellular decomposition of the pair $(\tilde{W},p^{-1}(L))$, so $H(CM(W,L;\Z[\pi_{1}(W)]))\cong H(\tilde{W}, p^{-1}(L); \Z)$. For $0 \leq * \leq \text{dim}(W)$ the subcomplex $(\mathcal{C}_{*}(f, \rho, J), d_{*})$ is precisely the  $\Z[\pi_{1}(W)]$-Morse complex of the pair $(W,L)$.
For $N_{W}> \text{dim}(W) +1$ differential $d_{\text{dim}(W)+1}$ and $d_{0}$ are both $0$ (since the moduli spaces are empty by dimension reason). This implies that $CM(W,L;\Z[\pi_{1}(W)]) \hookrightarrow \mathcal{C}(f, \rho, J;\Z[\pi_{1}(W)])$ induces an injective map $H(CM(W,L;\Z[\pi_{1}(W)])) \to QH(W,L,\Z[\pi_{1}(W)])$. On the other hand, we know that $QH(W,L,\Z[\pi_{1}(W)]) = 0$ since the cobordism is displaceable and there exist a PSS isomorphism between the Floer homology with local coefficients and the quantum homology with local coefficients \cite{Za}. From this we obtain $H_{*}(CM(W,L;\Z[\pi_{1}(W)])) = H_{*}(\tilde{W}, p^{-1}(L); \Z) = 0$. By hypothesis $L\hookrightarrow W$ induces an injective map on the fundamental group, since we previously have proved that $\pi_{1}(L)\rightarrow \pi_{1}(W)$ is surjective, then we have the $\pi_{1}(L)\rightarrow \pi_{1}(W)$ is an isomorphism, so $p^{-1}(L) = \tilde{L}$. From the long exact sequence in homology of the pair $(\tilde{W},\tilde{L} )$ and one of the Whitehead Theorems we obtain that the inclusion $L\hookrightarrow W$ is a homotopy equivalence. We proceed in the analogous way to show that $L'\hookrightarrow W$ is a homotopy equivalence.
\begin{remark}Alternatively, notice that from the proof above we have $H_{*}(\tilde{W}, p^{-1}(L); \Z) = 0$, so in particular $H_{0}(p^{-1}(L); \Z) \cong H_{0}(\tilde{W}; \Z) \cong \Z$. As pointed out to us by Baptiste Chantraine (see Remark 2) this implies that $\pi_{1}(L)\rightarrow \pi_{1}(W)$ is surjective.
\end{remark}

%

\bibliographystyle{alpha}
\bibliography{bibliography}
\end{document}